\documentclass{amsart} 
\usepackage{amscd}
\usepackage{amsmath,empheq}
\usepackage{amsfonts}
\usepackage{amssymb}
\usepackage{mathrsfs}
\usepackage[all]{xy}
\newtheorem{theorem}{Theorem}
\newtheorem{ex}{Example}
\newtheorem{lemma}[theorem]{Lemma}
\newtheorem{prop}[theorem]{Proposition}
\newtheorem{remark}{Remark}
\newtheorem{corollary}[theorem]{Corollary}

\newenvironment{proof-sketch}{\noindent{\bf Sketch of Proof}\hspace*{1em}}{\qed\bigskip}

\everymath{\displaystyle} 
\newcommand{\RR}{\mathbb R}
\newcommand{\NN}{\mathbb N}

\renewcommand{\leq}{\leqslant}

\renewcommand{\geq}{\geqslant}  
\begin{document}
\title[Nonlinear nonhomogeneous Robin problems]{Nodal solutions for nonlinear nonhomogeneous Robin problems}
\author[N.S. Papageorgiou]{Nikolaos S. Papageorgiou}
\address[N.S. Papageorgiou]{National Technical University, Department of Mathematics,
				Zografou Campus, Athens 15780, Greece \& Institute of Mathematics, Physics and Mechanics, Jadranska 19, 1000 Ljubljana, Slovenia}
\email{\tt npapg@math.ntua.gr}
\author[V.D. R\u{a}dulescu]{Vicen\c{t}iu D. R\u{a}dulescu}
\address[V.D. R\u{a}dulescu]{Institute of Mathematics, Physics and Mechanics, Jadranska 19, 1000 Ljubljana, Slovenia \& Faculty of Applied Mathematics, AGH University of Science and Technology, 30-059 Krak\'ow, Poland \& Department of Mathematics, University of Craiova, 200585 Craiova, Romania}
\email{\tt vicentiu.radulescu@imfm.si}
\author[D.D. Repov\v{s}]{Du\v{s}an D. Repov\v{s}}
\address[D.D. Repov\v{s}]{Faculty of Education and Faculty of Mathematics and Physics, University of Ljubljana \& Institute of Mathematics, Physics and Mechanics, Jadranska 19, 1000 Ljubljana, Slovenia}
\email{\tt dusan.repovs@guest.arnes.si}
\begin{abstract}
We consider the nonlinear Robin problem driven by a nonhomogeneous differential operator plus an indefinite potential. The reaction term is a Carath\'eodory function satisfying certain conditions only near zero. Using suitable truncation, comparison, and cut-off techniques, we show that the problem has a sequence of nodal solutions converging to zero in the $C^1(\overline{\Omega})$-norm.
\end{abstract}
\maketitle
\section{Introduction}
Let $\Omega\subseteq\RR^N$ be a bounded domain with a $C^2$-boundary $\partial\Omega$. We study the following nonlinear nonhomogeneous Robin problem:
\begin{equation}\label{eq1}
	\left\{
		\begin{array}{l}
			-{\rm div}\, a(Du(z)) + \xi(z)|u(z)|^{p-2}u(z) = f(z,u(z))\ \mbox{in}\ \Omega,\\
			\frac{\partial u}{\partial n_a} + \beta(z)|u|^{p-2}u=0\ \mbox{on}\ \partial\Omega.
		\end{array}
	\right\}
\end{equation}
In this problem, $a:\RR^N\rightarrow\RR^N$ is a continuous and strictly monotone map (thus also maximal monotone), which satisfies certain regularity and growth conditions listed in hypotheses $H(a)$ below. These conditions are general and they incorporate in our framework many differential operators of interest, such as the $p$-Laplacian and the $(p,q)$-Laplacian. We stress that $a(\cdot)$ is not homogeneous and this is a source of difficulties in the study of problem (\ref{eq1}). The potential function $\xi\in L^\infty(\Omega)$ is indefinite (that is, sign changing). The reaction term (the right-hand side of (\ref{eq1})) is a Carath\'eodory function (that is, for all $x\in\RR$, the function $z\mapsto f(z,x)$ is measurable, and for almost all $z\in\Omega$, the function $x\mapsto f(z,x)$) is continuous. We impose conditions on $f(z,\cdot)$ only near zero. In the boundary condition, $\frac{\partial u}{\partial n_a}$ denotes the conormal derivative corresponding to the differential operator $u\mapsto {\rm div}\,a(Du)$ and is defined by extension of the map
$$C^1(\overline{\Omega})\ni u\mapsto (a(Du),n)_{\RR^N},$$
with $n(\cdot)$ being the outward unit normal on $\partial\Omega$.

We are looking for nodal (that is, sign-changing) solutions for problem (\ref{eq1}). Employing a symmetry condition on $f(z,\cdot)$ near zero and using truncation, perturbation, comparison, and cut-off techniques, and a result of Kajikiya \cite{7}, we generate a whole sequence $\{u_n\}_{n\geq1}\subseteq C^1(\overline{\Omega})$ of distinct nodal solutions such that $u_n\rightarrow0$ in $C^1(\overline{\Omega})$.

The first result in this direction was produced by Wang \cite{26}, who used cut-off techniques to produce an infinity of solutions converging to zero in $H^1_0({\Omega})$. In Wang \cite{26} the problem is semilinear driven by the Dirichlet Laplacian. There is no potential term (that is, $\xi\equiv0$). The sequence produced by Wang \cite{26} does not consist of nodal solutions. More recently, Li \& Wang \cite{9} produced a sequence of nodal solutions for semilinear Schr\"{o}dinger equations. For nonlinear equations we mention the recent works of He, Huang, Liang \& Lei \cite{5}, and Papageorgiou \& R\u{a}dulescu \cite{18}. In He {\it et al.} \cite{5}, the problem is Neumann (that is, $\beta\equiv0$) and the differential operator is the $p$-Laplacian (that is, $a(y)=|y|^{p-2}y$ for all $y\in\RR^N$, with $1<p<\infty$). In Papageorgiou \& R\u{a}dulescu \cite{18}, the differential operator is the same as in the present paper, but $\xi\equiv0$. Also, the hypotheses on $f(z,\cdot)$ near zero are more restrictive. In the present paper we extend the results of all aforementioned works.

\section{Preliminaries and Hypotheses}

In the study of problem (\ref{eq1}) we will use the following spaces: the Sobolev space $W^{1,p}(\Omega)$, the Banach space $C^1(\overline{\Omega})$, and the boundary Lebesgue spaces $L^r(\partial\Omega)$, $1\leq r\leq\infty$.

We denote by $||\cdot||$  the norm on the Sobolev space $W^{1,p}(\Omega)$ defined by
$$
||u||=\left[||u||^p_p + ||Du||^p_p\right]^{\frac{1}{p}}\ \mbox{for all}\ u\in W^{1,p}(\Omega).
$$

The Banach space $C^1(\overline\Omega)$ is an ordered Banach space, with positive (order) cone $$C_+=\left\{u\in C^1(\overline{\Omega}):u(z)\geq0\ \mbox{for all}\ z\in\overline{\Omega}\right\}.$$ This cone has a nonempty interior which contains the open set
$$
D_+=\{u\in C_+:u(z)>0\ \mbox{for all}\ z\in\overline{\Omega}\}.
$$

In fact, $D_+$ is the interior of $C_+$ when furnished with the relative $C(\overline{\Omega})$-norm topology.

On $\partial\Omega$ we consider the $(N-1)$-dimensional Hausdorff (surface) measure $\sigma(\cdot)$. Using this measure, we can define in the usual way the Lebesgue spaces $L^r(\partial\Omega), 1\leq r\leq\infty$. From the theory of Sobolev spaces we know that there exists a unique continuous linear map $\gamma_0:W^{1,p}(\Omega)\rightarrow L^p(\partial\Omega)$, known as the ``trace map", such that
$$
\gamma_0(u)=u|_{\partial\Omega}\ \mbox{for all}\ u\in W^{1,p}(\Omega)\cap C(\overline{\Omega}).
$$

So, the trace map assigns ``boundary values" to all Sobolev functions. We know that the trace map is compact into $L^r(\partial\Omega)$ for all $1\leq r<\frac{(N-1)p}{N-p}$ if $p<N$, and into $L^r(\partial\Omega)$ for all $1\leq r<\infty$ if $p\geq N$. Furthermore, we have that
$$
{\rm ker}\,\gamma_0=W^{1,p}_0(\Omega)\ \mbox{and}\ {\rm im}\,\gamma_0=W^{\frac{1}{p'},p}(\partial\Omega)\quad \left(\frac{1}{p}+\frac{1}{p'}=1\right).
$$

In what follows, for the sake of notational simplicity, we will drop the use of the trace map $\gamma_0(\cdot)$. All restrictions of Sobolev functions on $\partial\Omega$, are understood in the sense of traces.

Let $X$ be a Banach space and $\varphi\in C^1(X,\RR)$. We say that $\varphi$ satisfies the ``Palais-Smale condition" (the ``PS-condition" for short), if the following property holds:
\begin{center}
``Every sequence $\{u_n\}_{n\geq1}\subseteq X$ such that\\
$\{\varphi(u_n)\}_{n\geq1}\subseteq\RR$ is bounded and
$\varphi'(u_n)\rightarrow0$ in $X^{*}$ as $n\rightarrow\infty$,\\
admits a strongly convergent subsequence."
\end{center}

\smallskip
We shall need the following result of Kajikya \cite{7}.
\begin{theorem}\label{th1}
	Assume that $X$ is a Banach space, $\varphi\in C^1(X,\RR)$ satisfies the PS-condition, $\varphi$ is even and bounded below, $\varphi(0)=0$, and for every $n\in\NN$, there exists an $n$-dimensional subspace $V_n$ of $X$ and $\rho_n>0$ such that
	$$
	\sup\{\varphi(u):u\in V_n\cap\partial B_{\rho_n}\}<0,
	$$
	where $\partial B_{\rho_n}=\{u\in X:||u||_X=\rho_n\}$. Then there exists a sequence $\{u_n\}_{n\geq1}\subseteq X\backslash\{0\}$ such that
	$$
	\begin{array}{lll}
	&(i)&\	\varphi'(u_n)=0\ \mbox{for all}\ n\in\NN\ \mbox{(that is, each $u_n$ is a critical point of $\varphi$)}; \\
	&(ii)&\	\varphi(u_n)\leq 0\ \mbox{for all}\ n\in\NN;\ and \\
	&(iii)&\	u_n\rightarrow 0\ \mbox{in}\ X\ \mbox{as}\ n\rightarrow\infty.
	\end{array}
	$$
\end{theorem}

In the sequel, for any $\varphi\in C^1(X,\RR)$, we denote by $K_\varphi$ the critical set of $\varphi$, that is,
$$
K_\varphi=\{u\in X:\varphi'(u)=0\}.
$$

For $X\in\RR$, we set $x^{\pm}=\max\{\pm x,0\}$. Then for any $u\in W^{1,p}(\Omega)$, we define $u^{\pm}(\cdot)=u(\cdot)^{\pm}$. We know that
$$
u^{\pm}\in W^{1,p}(\Omega),\ u=u^+-u^-,\ |u|=u^++u^-.
$$

Let $\vartheta\in C^1(0,\infty)$ be such that $\vartheta(t)>0$ for all $t>0$ and
\begin{equation}\label{eq2}
	\begin{array}{rr}
		0<\hat{c}\leq\frac{\vartheta'(t)t}{\vartheta(t)}\leq c_0\ \mbox{and}\ c_1t^{p-1}\leq\vartheta(t)\leq c_2(t^{\tau-1}+t^{p-1}) \\
		\mbox{for all}\ t>0,\ \mbox{with}\ c_1,c_2>0, 1\leq\tau<p.	
	\end{array}
\end{equation}

Then the hypotheses on the map $a(\cdot)$ are the following:

\smallskip
$H(a):$ $a(y)=a_0(|y|)y$ for all $y\in\RR^N$ with $a_0(t)>0$ for all $t>0$ and
\begin{itemize}
	\item [(i)] $a_0\in C^1(0,\infty),\ t\mapsto a_0(t)t$ is strictly increasing on $(0,\infty)$, $a_0(t)t\rightarrow0^+$ as $t\rightarrow0^+$ and
		$$
		\lim_{t\rightarrow0^+}\frac{a_0'(t)t}{a_0(t)}>-1;
		$$
	\item [(ii)] $|\nabla a(y)|\leq c_3\frac{\vartheta(|y|)}{|y|}$ for all $y\in\RR^N\backslash\{0\}$, and some $c_3>0$;
	\item [(iii)] $(\nabla a(y)\xi,\xi)_{\RR^N}\geq\frac{\vartheta(|y|)}{|y|}|\xi|^2$ for all $y\in\RR^N\backslash\{0\}$, $\xi\in\RR^N$; and
	\item [(iv)] If $G_0(t)=\int_0^t a_0(s)sds$ for all $t>0$, then there exists $q\in(1,p]$ such that
		$$
			t\mapsto G_0(t^{1/q})\ \mbox{is convex and }
			\limsup_{t\rightarrow0^+}\frac{qG_0(t)}{t^q}<+\infty.
		$$
\end{itemize}
\begin{remark}
	Hypotheses $H(a)(i),(ii),(iii)$ are dictated by the nonlinear global regularity theory of Lieberman \cite{8} and the nonlinear maximum principle of Pucci \& Serrin \cite{23}. Hypothesis $H(a)(iv)$ reflects the particular requirements of our problem. However, $H(a)(iv)$ is not restrictive as the examples below illustrate.
\end{remark}

Hypotheses $H(a)$ imply that $G_0(\cdot)$ is strictly convex and strictly increasing. We set $G(y)=G_0(|y|)$ for all $y\in\RR^N$. Evidently, $G(\cdot)$ is convex and $G(0)=0$. Also, we have
$$
\nabla G(y)=G'_0(|y|)\frac{y}{|y|}=a_0(|y|)y=a(y)\ \mbox{for all}\ y\in\RR^N\backslash\{0\},\ \nabla G(0)=0.
$$

So, $G(\cdot)$ is the primitive of $a(\cdot)$. Moreover, the convexity of $G(\cdot)$ implies that
\begin{equation}\label{eq3}
	G(y)\leq(a(y),y)_{\RR^N}\ \mbox{for all}\ y\in\RR^N.
\end{equation}

The next lemma summarizes the main properties of the map $a(\cdot)$ and it is an easy consequence of hypotheses $H(a)$ and condition (\ref{eq2}) above.
\begin{lemma}\label{lem2}
	If hypotheses $H(a)(i),(ii),(iii)$ hold, then
	\begin{itemize}
		\item [(a)] $a(\cdot)$ is continuous, strictly monotone, hence maximal monotone, too;
		\item [(b)] $|a(y)|\leq c_4(1+|y|^{p-1})$ for all $y\in\RR^N$, and some $c_4>0$; and
		\item [(c)] $(a(y),y)_{\RR_N}\geq\frac{c_1}{p-1}|y|^p$ for all $y\in\RR^N$.
	\end{itemize}
\end{lemma}
This lemma and (\ref{eq3}) lead to the following growth conditions on $G(\cdot)$.
\begin{corollary}\label{cor3}
	If hypotheses $H(a)(i),(ii),(iii)$ hold, then $\frac{c_1}{p(p-1)}|y|^p\leq G(y)\leq c_5(1+|y|^p)$ for all $y\in\RR^N$, and some $c_5>0$.
\end{corollary}
\begin{ex}
	The following maps $a(y)$ satisfy hypotheses $H(a)$:
	\begin{itemize}
		\item [(a)] $a(y)=|y|^{p-2}y, 1<p<\infty$.\\
			This map corresponds to the p-Laplace differential operator defined by
			$$
			\Delta_p u={\rm div}\,(|Du|^{p-2}Du)\ \mbox{for all}\ u\in W^{1,p}(\Omega).
			$$
		\item [(b)] $a(y)=|y|^{p-2}y + |y|^{q-2}y, 1<q<p<\infty$.
			This map corresponds to the $(p,q)$-Laplace differential operator defined by
			$$
			\Delta_p u+\Delta_q u\ \mbox{for all}\ u\in W^{1,p}(\Omega).
			$$
			
			Such operators arise in problems of mathematical physics. Recently $(p,q)$-equations have been studied by Bobkov \& Tanaka \cite{1}, Li \& Zhang \cite{10}, Marano \& Mosconi \cite{11}, Marano, Mosconi \& Papageorgiou \cite{12,13}, Mugnai \& Papageorgiou \cite{15}, Papageorgiou \& R\u{a}dulescu \cite{16}, Sun, Zhang \& Su \cite{24},
			and  Tanaka \cite{25}.
		\item [(c)] $a(y)=(1+|y|^2)^{\frac{p-2}{2}}y, 1<p<\infty$.
			This map corresponds to the generalized p-mean curvature differential operator defined by
			$$
			{\rm div}\,((1+|Du|^2)^{\frac{p-2}{2}}Du)\ \mbox{for all}\ u\in W^{1,p}(\Omega).
			$$
		\item [(d)] $a(y)=|y|^{p-2}y(1+\frac{1}{1+|y|^p}), 1<p<\infty$.
	\end{itemize}
\end{ex}

We denote by $\langle \cdot,\cdot\rangle$ the duality brackets for the pair
$$
(W^{1,p}(\Omega)^*,\ W^{1,p}(\Omega)).
$$

Let $A:W^{1,p}(\Omega)\rightarrow W^{1,p}(\Omega)^*$ be the nonlinear map defined by
$$
\langle A(u),h\rangle = \int_\Omega(a(Du),Dh)_{\RR^N}dz\ \mbox{for all}\ u,h\in W^{1,p}(\Omega).
$$

From Gasinski \& Papageorgiou \cite{3}, we have:
\begin{prop}\label{prop4}
	The map $A:W^{1,p}(\Omega)\rightarrow W^{1,p}(\Omega)^*$ is bounded (maps bounded sets to bounded sets), continuous, monotone (hence maximal monotone, too), and of type $(S)_+$, that is,
	$$
	``u_n\xrightarrow{w}u\ \mbox{in}\ W^{1,p}(\Omega)\ \mbox{and}\ \limsup_{n\rightarrow\infty}\langle A(u_n),u_n-u\rangle\Rightarrow u_n\rightarrow u".
	$$
\end{prop}

The hypotheses on the potential function $\xi(\cdot)$ and on the boundary coefficient $\beta(\cdot)$ are the following:

\smallskip
$H(\xi): \xi\in L^\infty(\Omega)$.

\smallskip
$H(\beta): \beta\in C^{0,\alpha}(\partial\Omega)\ \mbox{for some}\ \alpha\in(0,1)\ \mbox{and}\ \beta(z)\geq0\ \mbox{for all}\ z\in\partial\Omega.$

\begin{remark}
	If $\beta\equiv0$, then we recover the Neumann problem.
\end{remark}

Finally, we introduce our conditions on the reaction term $f(z,x)$:

\smallskip
$H(f)$: $f:\Omega\times\RR\rightarrow\RR$ is a Carath\'eodory function such that $f(z,0)=0$ for almost all $z\in\Omega$ and
\begin{itemize}
	\item [(i)] there exists $\eta>0$ such that for almost all $z\in\Omega, f(z,\cdot)|_{[-\eta,\eta]}$ is odd;
	\item [(ii)] $|f(z,x)|\leq a_\eta(z)$ for almost all $z\in\Omega$,  $x\in[-\eta,\eta]$, with $a_\eta\in L^\infty(\Omega)$;
	\item [(iii)] with $q\in(1,p]$ as in hypothesis $H(a)(iv)$, we have
		$$
		\lim_{x\rightarrow0}\frac{f(z,x)}{|x|^{q-2}x} = +\infty\ \mbox{uniformly for almost all}\ z\in\Omega;\ \mbox{and}
		$$
	\item [(iv)] there exists $\hat{\xi}>0$ such that for almost all $z\in\Omega$
		$$
		x\rightarrow f(z,x) + \hat{\xi}|x|^{p-2}x
		$$
		is nondecreasing on $[-\eta,\eta]$
\end{itemize}
\begin{remark}
	We point out that all the above hypotheses concern the behaviour of $f(z,\cdot)$ only near zero.
\end{remark}

Finally, we mention that nonlinear problems with an indefinite potential have recently been studied in the context of equations driven by the Neumann $p$-Laplacian by Gasinski \& Papageorgiou \cite{4} (resonant problems) and Fragnelli, Mugnai \& Papageorgiou \cite{2} (superlinear problems). Also, nodal solutions for nonlinear Robin problems with no potential term, were obtained by Papageorgiou \& R\u{a}dulescu \cite{20}.

\section{Nodal solutions}

Let $\varepsilon\in(0,\eta)$ and consider an even function $\gamma\in C^1(\RR)$ such that $0\leq\gamma\leq 1$, $\gamma|_{[-\varepsilon,\varepsilon]}=1$ and ${\rm supp}\,\gamma\subseteq[-\eta,\eta]$.

We set
$$
\hat{f}(z,x)=\gamma(x)f(z,x) + (1-\gamma(x))\xi(z)|x|^{p-2}x.
$$

Evidently, $\hat{f}(z,x)$ is a Carath\'eodory function which is odd in $x\in\RR$ and has the following two additional properties:
\begin{eqnarray}
\hat{f}(z,\cdot)|_{[-\varepsilon,\varepsilon]}=f(z,\cdot)|_{[-\varepsilon,\varepsilon]}\ \mbox{for all}\ z\in\Omega; \label{eq4}\\
\hat{f}(z,x)=\xi(z)|x|^{p-2}x\ \mbox{for all}\ z\in\Omega,\  |x|\geq\eta. \label{eq5}
\end{eqnarray}

It follows from (\ref{eq5})  that
\begin{equation}\label{eq6}
	\hat{f}(z,\eta)-\xi(z)\eta^{p-1}=0\ \mbox{for almost all}\ z\in\Omega.
\end{equation}

Since $\hat{f}(z,\cdot)$ is odd, we have
\begin{equation}\label{eq7}
	\hat{f}(z,-\eta) + \xi(z)\eta^{p-1}=0\ \mbox{for almost all}\ z\in\Omega.
\end{equation}

On account of hypothesis $H(f)(iii)$, given any $\mu>0$, we can find $\delta=\delta(\mu)\in(0,\varepsilon)$ such that
\begin{equation}\label{eq8}
	f(z,x)x=\hat{f}(z,x)\geq\mu|x|^q\ \mbox{for almost all}\ z\in\Omega,\ \mbox{and all}\ |x|\leq\delta\ \mbox{(see (\ref{eq4})).}
\end{equation}

Then (\ref{eq8}) combined with hypothesis $H(f)(ii)$ implies that given $r>p$ we can find $c_6>0$ such that
\begin{equation}\label{eq9}
	\hat{f}(z,x)x\geq\mu|x|^q-c_6|x|^r\ \mbox{for almost all}\ z\in\Omega,\ \mbox{and all}\ x\in\RR.
\end{equation}

We introduce the following function
\begin{equation}\label{eq10}
	k(z,x)=\mu|x|^{q-2}x-c_6|x|^{r-2}x.
\end{equation}

This is a Carath\'eodory function which is odd in $x\in\RR$.

We consider the following auxiliary nonlinear Robin problem:
\begin{equation}\label{eq11}
	\left\{
		\begin{array}{l}
			-{\rm div}\,a(Du(z))+|\xi(z)||u(z)|^{p-2}u(z)=k(z,u(z))\ \mbox{in}\ \Omega, \\
			\frac{\partial u}{\partial n_a} + \beta(z)|u|^{p-2}u=0\ \mbox{on}\ \partial\Omega.
		\end{array}
	\right\}
\end{equation}
\begin{prop}\label{prop5}
	If hypotheses $H(a),\ H(\xi),\ H(\beta)$ hold, then problem (\ref{eq11}) admits a unique positive solution
	$$
	u^*\in D_+
	$$
	and since $k(z,\cdot)$ is odd, $v^*=-u^*\in D_+$ is the unique negative solution of (\ref{eq11}).
\end{prop}
\begin{proof}
	We consider the Carath\'eodory function $\hat{k}(z,x)$ defined by
	\begin{equation}\label{eq12}
		\hat{k}(z,x)=\left\{
	     \begin{array}{ll}
	       k(z,-\eta)-\eta^{p-1} & \mbox{if}\ x<-\eta\\
	       k(z,x) + |x|^{p-2}x & \mbox{if}\ -\eta\leq x\leq\eta \\
	       k(z,\eta) + \eta^{p-1}\ & \mbox{if}\ \eta<x.
	     \end{array}
		\right.
	\end{equation}
	
	We set $\hat{K}(z,x)=\int_0^x\hat{k}(z,s)ds$ and consider the $C^1$-functional $\hat{\varphi_+}: W^{1,p}(\Omega)\rightarrow\RR$ defined by
	$$
	\begin{array}{rr}
		\hat{\varphi_+}(u)=\int_\Omega G(Du)dz + \frac{1}{p}\int_\Omega [|\xi(z)|+1]|u|^pdz + \frac{1}{p}\int_{\partial\Omega}\beta(z)|u|^pd\sigma \\ - \int_\Omega\hat{K}(z,u^+)dz\quad
		\mbox{for all}\ u\in W^{1,p}(\Omega).
	\end{array}
	$$
	
	From (\ref{eq12}) and Corollary \ref{cor3} it is clear that
	$$
	\hat{\varphi}_+(\cdot)\ \mbox{is  coercive.}
	$$
	
	Also, from the Sobolev embedding theorem and the compactness of the trace map, we deduce that
	$$
	\hat{\varphi}_+(\cdot)\ \mbox{is sequentially weakly lower semicontinuous.}
	$$
	
	So, by the Weierstrass-Tonelli theorem, we can find $u^*\in W^{1,p}(\Omega)$ such that
	\begin{equation}\label{eq13}
		\hat{\varphi}_+(u^*)=\inf\left\{\hat{\varphi}_+(u): u\in W^{1,p}(\Omega)\right\}.
	\end{equation}
	
	On account of hypothesis $H(a)(iv)$, we can find $c_7>0$ such that
	\begin{equation}\label{eq14}
		G(y)\leq\frac{c_7}{q}|y|^q\ \mbox{for all}\ |y|\leq\delta,
	\end{equation}
	with $\delta>0$ as in (\ref{eq8}). Let $u\in D_+$. Then we can find $t\in(0,1)$ small such that
	\begin{equation}\label{eq15}
		tu(z)\in(0,\delta]\ \mbox{and}\ |D(tu)(z)|\leq\delta\ \mbox{for all}\ z\in\overline{\Omega}.
	\end{equation}
	
	Using (\ref{eq10}), (\ref{eq12}), (\ref{eq14}) and (\ref{eq15}), we obtain
	$$
	\begin{array}{ll}
		\hat{\varphi}_+(tu) & \leq \frac{t^qc_7}{q}||Du||^q_q + \frac{t^q}{q}\int_\Omega|\xi(z)||u|^qdz + \frac{t^q}{q}\int_{\partial\Omega}\beta(z)|u|^qd\sigma\\ + \frac{t^r}{r}||u||^r_r - \frac{t^q}{q}\mu||u||^q_q\\
		& \mbox{(since $t\in(0,1),\ q\leq p<r$)}\\
		& \leq \left[ c_8-\mu c_9\right]t^q\ \mbox{for some}\ c_8, c_9>0\ \mbox{depending on $u$}.
	\end{array}
	$$
	
	Choosing $\mu>\frac{c_8}{c_9}$, we infer that
	$$
	\begin{array}{ll}
		& \hat{\varphi}_+(tu)<0,\\
		\Rightarrow & \hat{\varphi}_+(u^*)<0=\hat{\varphi}_+(0)\ \mbox{(see (\ref{eq13}))},\\
		\Rightarrow & u^*\neq0.
	\end{array}
	$$
	
	From (\ref{eq13}) we have
	\begin{eqnarray}\label{eq16}
			&& \hat{\varphi}_+'(u^*)=0,\nonumber\\
			&\Rightarrow &\langle A(u^*),h\rangle + \int_\Omega[|\xi(z)|+1]|u^*|^{p-2}u^*hdz + \int_{\partial\Omega}\beta(z)|u^*|^{p-2}u^*hdz = \nonumber\\
			&&\hspace{1cm}\int_\Omega\hat{k}(z,(u^*)^+)hdz\\
			&& \mbox{for all}\ h\in W^{1,p}(\Omega).\nonumber
	\end{eqnarray}
	
	In (\ref{eq16}) we choose $h=-(u^*)^-\in W^{1,p}(\Omega)$. Using Lemma \ref{lem2}(c), we obtain
	$$
	\begin{array}{ll}
		& \frac{c_1}{p-1}||D(u^*)^-||^p_p + ||(u^*)^-||^p_p\leq0\ \mbox{(see hypothesis $H(B)$)},\\
		\Rightarrow & u^*\geq0,\ u^*\neq0.	
	\end{array}
	$$
	
	In (\ref{eq16}) we choose $h=(u^*-\eta)^+\in W^{1,p}(\Omega)$. Then
	$$
	\begin{array}{ll}
		& \langle A(u^*),(u^*-\eta)^+\rangle + \int_\Omega[|\xi(z)|+1](u^*)^{p-1}(u^*-\eta)^+dz \\ + & \int_{\partial\Omega}\beta(z)(u^*)^{p-1}(u^*-\eta)^+d\sigma \\
		= & \int_\Omega\left[ \mu\eta^{q-1} - c_6\eta^{r-1} + \eta^{p-1}\right](u^*-\eta)^+dz\ \mbox{(see (\ref{eq12}) and (\ref{eq10}))} \\
		\leq & \int_\Omega\left[\hat{f}(z,\eta)+\eta^{p-1}\right](u^*-\eta)^+dz\ \mbox{(see (\ref{eq9}))}\\
		= & \int_\Omega\left[ \xi(z)+1\right]\eta^{p-1}(u^*-\eta)^+dz\ \mbox{(see (\ref{eq6}))}\\
		\leq & \langle A(\eta),(u^*-\eta)^+\rangle + \int_\Omega\left[|\xi(z)|+1\right]\eta^{p-1}(u^*-\eta)^+dz + \int_{\partial\Omega}\beta(z)\eta^{p-1}(u^*-\eta)^+d\sigma\\
		& \mbox{(note that $A(\eta)=0$ and see hypothesis $H(\beta)$),}\\
		\Rightarrow & \langle A(u^*)-A(\eta),(u^*-\eta)^+\rangle + \int_{\Omega}\left[|\xi(z)|+1\right]((u^*)^{p-1}-\eta^{p-1})(u^*-\eta)^+dz\leq0\\
		& \mbox{(see hypotesis $H(\beta)$)},\\
		\Rightarrow & u^*\leq\eta.
	\end{array}
	$$
	
	So, we have proved that
	\begin{equation}\label{eq17}
		u^*\in[0,\eta]=\left\{u\in W^{1,p}(\Omega):0\leq u(z)\leq\eta\ \mbox{for almost all}\ z\in\Omega\right\}.
	\end{equation}
	
	From (\ref{eq10}), (\ref{eq12}), (\ref{eq16}) and (\ref{eq17}), we infer that $u^*$ is a positive solution of problem (\ref{eq11}). From Papageorgiou \& R\u{a}dulescu \cite{19}, we have
	$$
	u^*\in L^\infty(\Omega).
	$$
	
	Now the nonlinear regularity theory of Lieberman \cite{8} implies that
	$$
	u^*\in C_+\backslash\{0\}.
	$$
	
	From (\ref{eq16}) and (\ref{eq17}), we have
	$$
	\begin{array}{ll}
		& \left\{
			\begin{array}{ll}
					-{\rm div}\,a(Du^*(z)) + |\xi(z)|u^*(z)^{p-1} = k(z,u^*(z))\ \mbox{for almost all}\ z\in\Omega, \\
					\frac{\partial u^*}{\partial n_a} + \beta(z)u^*=0\ \mbox{on}\ \partial\Omega
			\end{array}
		\right\} \\
		& \mbox{(see Papageorgiou \& R\u{a}dulescu \cite{17})} \\
		\Rightarrow & -{\rm div}\,a(Du^*(z)) + |\xi(z)|u^*(z)^{p-1}\geq-c_6u^*(z)^{r-1}\ \mbox{for almost all}\ z\in\Omega\ \mbox{(see (\ref{eq10}))}, \\
		\Rightarrow & {\rm div}\,a(Du^*(z))\leq\left[c_6||u^*||^{r-p}_\infty + ||\xi||_\infty\right]u^*(z)^{p-1}\ \mbox{for almost all}\ z\in\Omega\ \\
		& \mbox{(see hypothesis $H(\xi)$)}, \\
		\Rightarrow & u^*\in D_+\ \mbox{(see Pucci \& Serrin \cite[p. 120]{23})}.
	\end{array}
	$$
	
	Next, we show the uniqueness of this solution. To this end, let $\hat{i}:L^1(\Omega)\rightarrow\overline{\RR}=\RR\cup\{+\infty\}$ be the integral functional defined by
	$$
	\hat{i}(u)=\left\{
		\begin{array}{ll}
			\int_\Omega G(Du^{\frac{1}{q}})dz + \frac{1}{p}\int_\Omega|\xi(z)|u^{\frac{p}{q}}dz + \frac{1}{p}\int_{\partial\Omega}\beta(z)u^\frac{p}{q}d\sigma & \mbox{if}\ u\geq0, u^\frac{1}{q}\in W^{1,p}(\Omega) \\
			+\infty & \mbox{otherwise}.
		\end{array}
	\right.
	$$
	
	From Papageorgiou \& Winkert \cite{22} (see the proof of Proposition 3.3), we know that $\hat{i}(\cdot)$ is convex and if $u^*,v^*\in D_+$ are two positive solutions of (\ref{eq11}), then
	$$
	\begin{array}{ll}
		\hat{i}'((u^*)^q)(h)= & \frac{1}{q}\int_\Omega\frac{-{\rm div}\,a(Du^*)+|\xi(z)|(u^*)^{p-1}}{(u^*)^{q-1}}hdz \\
		\hat{i}'((v^*)^q)(h)= & \frac{1}{q}\int_\Omega\frac{-{\rm div}\,a(Dv^*)+|\xi(z)|(v^*)^{p-1}}{(v^*)^{q-1}}hdz\ \mbox{for all}\ h\in C^1(\overline\Omega).
	\end{array}
	$$
	
	The convexity of $\hat{i}(\cdot)$ implies the monotonicity of $\hat{i}'(\cdot)$. Hence
	$$
	\begin{array}{ll}
		0 & \leq\int_\Omega\left[\frac{-{\rm div}\,a(Du^*)+|\xi(z)|(u^*)^{p-1}}{(u^*)^{q-1}} - \frac{{\rm div}\,a(Dv^*)+|\xi(z)|(v^*)^{p-1}}{(v^*)^{q-1}}\right]((u^*)^q-(v^*)^q)dz\\
		& = \int_\Omega c_6\left[(v^*)^{r-q}-(v^*)^{r-q}\right]((u^*)^q-(v^*)^q)dz\ \mbox{(see (\ref{eq10}))}, \\
		\Rightarrow & u^*=v^*\ \mbox{(since $q\leq p<r$)}.
	\end{array}
	$$
	
	This proves the uniqueness of the positive solution $u^*\in D_+$ of (\ref{eq11}). Since problem (\ref{eq11}) is odd, it follows that $v^*=-u^*\in-D_+$ is the unique negative solution of problem (\ref{eq11}).
\end{proof}

Consider the following Robin problem:
\begin{equation}\label{eq18}
	\left\{
		\begin{array}{ll}
			-{\rm div}\,a(Du(z)) + \xi(z)|u(z)|^{p-2}u(z) = \hat{f}(z,u(z))\ \mbox{in}\ \Omega, \\
			\frac{\partial u}{\partial n_a}+\beta(z)|u|^{p-2}u=0\ \mbox{on}\ \partial\Omega.
		\end{array}
	\right\}
\end{equation}

We denote by $S^+$ (respectively $S^-$)  the set of positive (respectively negative) solutions of problem (\ref{eq18}) which are in the order interval $[0,\eta]=\{u\in W^{1,p}(\Omega):0\leq u(z)\leq\eta$ for almost all $z\in\Omega\}$ (respectively in $[-\eta,0]=\left\{\right.v\in W^{1,p}(\Omega):-\eta\leq v(z)\leq 0$ for almost all $z\in\Omega\left.\right\}$). From Papageorgiou, R\u{a}dulescu \& Repov\v{s} \cite{21}, we know that
\begin{itemize}
	\item $S^+$ is downward directed (that is, if $u_1,\ u_2\in S^+$, then we can find $u\in S^+$ such that $u\leq u_1,u\leq u_2$).
	\item $S^-$ is upward directed (that is, if $v_1,\ v_2\in S^-$, then we can find $v\in S^-$ such that $v_1\leq v, v_2\leq v$).
\end{itemize}

Moreover, reasoning as in the proof of Proposition \ref{prop5} (with $k(z,x)$ replaced by $\hat{f}(z,x)$), we show that
$$
\emptyset\neq S^+\subseteq D_+\ \mbox{and}\ \emptyset\neq S^-\subseteq-D_+.
$$
\begin{prop}\label{prop6}
	If hypotheses $H(a),H(\xi),H(\beta),H(f)$ hold, then $u^*\leq u$ for all $u\in S^+$ and $v\leq v^*$ for all $v\in S^-$.
\end{prop}
\begin{proof}
	Let $u\in S_+$ and let $\hat{k}(z,x)$ be given by (\ref{eq12}). We introduce the following truncation of $\hat{k}(z,\cdot)$:
	\begin{equation}\label{eq19}
		e_+(z,x)=\left\{
			\begin{array}{ll}
				0 & \mbox{if}\ x<0 \\
				\hat{k}(z,x) & \mbox{if}\ 0\leq x\leq u(z) \\
				\hat{k}(z,u(z)) & \mbox{if}\ u(z)<x.
			\end{array}
		\right.
	\end{equation}
	
	This is a Carath\'eodory function. We set $E_+(z,x)=\int^x_0e_+(z,s)ds$ and consider the $C^1$-functional $\Psi_+:W^{1,p}(\Omega)\rightarrow\RR$ defined by
	$$
	\begin{array}{rr}
		\Psi_+(u)=\int_\Omega G(Du)dz + \frac{1}{p}\int_\Omega[|\xi(z)|+1]|u|^pdz + \frac{1}{p}\int_{\partial\Omega}\beta(z)|u|^pd\sigma - \int_\Omega E_+(z,u)dz \\
		\mbox{for all}\ u\in W^{1,p}(\Omega).
	\end{array}
	$$
	
	Evidently, $\Psi_+(\cdot)$ is coercive (see (\ref{eq19})) and sequentially weakly lower semicontinuous. So, we can find $\hat{u}^*\in W^{1,p}(\Omega)$ such that
	\begin{equation}\label{eq20}
		\Psi_+(\hat{u}^*)=\inf\left\{\Psi_+(u):u\in W^{1,p}(\Omega)\right\}.
	\end{equation}
	
	As in the proof of Proposition \ref{prop5}, using hypotheses $H(a)(iv)$ and $H(f)(iii)$, we show that
	$$
	\begin{array}{ll}
		& \Psi_+(\hat{u}^*)<0 = \Psi_+(0), \\
		\Rightarrow & \hat{u}^*\neq0.
	\end{array}
	$$
	
	From (\ref{eq20}) we have
	\begin{eqnarray}\label{eq21}
			&&\Psi_+'(\hat{u}^*)=0,\nonumber \\
			&\Rightarrow & \langle A(\hat{u}^*),h\rangle + \int_\Omega[|\xi(z)|+1]|\hat{u}^*|^{p-2}\hat{u}^*hdz + \int_{\partial\Omega}\beta(z)|\hat{u}^*|^{p-2}\hat{u}^*hd\sigma =\nonumber\\
			&&\hspace{1cm} \int_\Omega e_+(z,\hat{u}^*)hdz\ \mbox{for all}\ h\in W^{1,p}(\Omega).
	\end{eqnarray}
	
	In (\ref{eq21}), we choose $h=-(u^*)^-\in W^{1,p}(\Omega)$. Then using Lemma \ref{lem2}(c), we have
	$$
	\begin{array}{ll}
		& \frac{c_1}{p-1}||D(\hat{u}^*)^-||^p_p + \int_\Omega[|\xi(z)+1|]((\hat{u}^*)^-)^pdz\leq0\ \mbox{(see hypothesis $H(\beta)$ and (\ref{eq19}))} \\
		\Rightarrow & \hat{u}^*\geq0, \ \hat{u}^*\neq0.
	\end{array}
	$$
	
	Next, in (\ref{eq21}) we choose $h=(\hat{u}^*-u)^+\in W^{1,p}(\Omega)$. We have
	$$
	\begin{array}{ll}
		& \langle A(\hat{u}^*),(\hat{u}^*-u)^+\rangle + \int_\Omega[|\xi(z)+1|](\hat{u}^*)^{p-1}(\hat{u}^*-u)^+dz + \int_{\partial\Omega}\beta(z)(u^*)^{p-1}(\hat{u}^*-u)^+d\sigma \\
		= & \int_\Omega[\mu u^{q-1}-c_6 u^{r-1}+u^{p-1}](\hat{u}^*-u)^+dz\ \mbox{(see (\ref{eq19}), (\ref{eq12}), (\ref{eq10}) and recall that $u\in S^+$)} \\
		\leq & \int_\Omega[\hat{f}(z,u)+u^{p-1}](\hat{u}^*-u)^+dz\ \mbox{(see (\ref{eq9}))} \\
		= & \langle A(u),(\hat{u}^*-u)^+\rangle + \int_\Omega[|\xi(z)+1|]u^{p-1}(\hat{u}^*-u)^+dz + \int_{\partial\Omega}\beta(z)u^{p-1}(\hat{u}^*-u)^+d\sigma\ \\
		& \mbox{(since $u\in S^+$)} \\
		\Rightarrow & \hat{u}^*\leq u.
	\end{array}
	$$
	
	So, we have proved that
	$$
	\hat{u}^*\in[0,u]=\{y\in W^{1,p}(\Omega):0\leq y(z)\leq u(z) \ \mbox{for almost all}\ z\in\Omega\}.
	$$
	
	This fact, together with (\ref{eq10}), (\ref{eq12}), (\ref{eq19}), (\ref{eq21}), imply that
	$$
	\begin{array}{ll}
		& -{\rm div}\,a(D\hat{u}^*z) + |\xi(z)|\hat{u}^*(z)^{p-1} = k(z,\hat{u}^*(z))\ \mbox{for almost all}\ z\in\Omega, \\
		& \frac{\partial\hat{u}^*}{\partial n_a} + \beta(z)(\hat{u}^*)^{p-1}=0\ \mbox{on}\ \partial\Omega\ \mbox{(see Papageorgiou  \& R\u{a}dulescu \cite{17})}, \\
		\Rightarrow & \hat{u}^*=u^*\ \mbox{(see Proposition \ref{prop5})}, \\
		\Rightarrow & u^*\leq u\ \mbox{for all}\ u\in S^+.
	\end{array}
	$$
	
	Similarly, we show that
	$$
	v\leq v^*\ \mbox{for all}\ v\in S^-.
	$$
This completes the proof.
\end{proof}

Now we can establish the existence of extremal constant sign solutions for problem (\ref{eq18}), that is, we show that problem (\ref{eq18}) has a smallest positive solution and a biggest negative solution.
\begin{prop}\label{prop7}
	If hypotheses $H(a), H(\beta), H(\xi), H(f)$ hold, then there exists a smallest positive solution $u_+\in S^+\subseteq D_+$ and a biggest negative solution $v_+\in S^-\subseteq-D_+$.
\end{prop}
\begin{proof}
	Invoking Lemma 3.10 of Hu \& Papageorgiou \cite[p. 178]{6}, we can find a decreasing sequence $\{u_n\}_{n\geq1}\subseteq S^+$ such that
	$$
	\inf S^+=\inf_{n\geq1}u_n.
	$$
	
	Evidently, $\{u_n\}_{n\geq1}\subseteq W^{1,p}(\Omega)$ is bounded. So, we may assume that
	\begin{equation}\label{eq22}
		u_n\xrightarrow{w}u_+\ \mbox{in}\ W^{1,p}(\Omega)\ \mbox{and}\ u_n\rightarrow u_+\ \mbox{in}\ L^p(\Omega)\ \mbox{and}\ L^p(\partial\Omega).
	\end{equation}
	
	We have
	\begin{equation}\label{eq23}
		\begin{array}{r}
			\langle A(u_n),h\rangle + \int_\Omega\xi(z)u^{p-1}_nhdz + \int_{\partial\Omega}\beta(z)u^{p-1}_nhd\sigma = \int_\Omega\hat{f}(z,u_n)hdx \\
			\mbox{for all}\ h\in W^{1,p}(\Omega), n\in\NN.
		\end{array}
	\end{equation}
	
	In (\ref{eq23}) we choose $h=u_n-u_+\in W^{1,p}(\Omega)$, pass to the limit as $n\rightarrow\infty$ and use (\ref{eq22}). Then
	\begin{equation}\label{eq24}
		\begin{array}{ll}
			& \lim_{n\rightarrow\infty}\langle A(u_n),u_n-u_+\rangle = 0, \\
			\Rightarrow & u_n\rightarrow u_+\ \mbox{in}\ W^{1,p}(\Omega)\ \mbox{(see Proposition \ref{prop4})}.
		\end{array}
	\end{equation}
	
	In (\ref{eq23}) we pass to the limit as $n\rightarrow\infty$ and use (\ref{eq24}). Then
	\begin{equation}\label{eq25}
		\langle A(u_+),h\rangle + \int_\Omega\xi(z)u_+^{p-1}hdz + \int_{\partial\Omega}\beta(z)u_+^{p-1}hd\sigma = \int_\Omega\hat{f}(z,u_+)hdz\ \mbox{for all}\ h\in W^{1,p}(\Omega).
	\end{equation}
	
	From Proposition \ref{prop6}, we have
	\begin{equation}\label{eq26}
		\begin{array}{ll}
			& u^*\leq u_n\ \mbox{for all}\ n\in\NN, \\
			\Rightarrow & u^*\leq u_+\ \mbox{(see (\ref{eq24})), hence}\ u_+\neq0.
		\end{array}
	\end{equation}
	
	It follows from (\ref{eq25}) and (\ref{eq26}) that
	$$
	u_+\in S^+\subseteq D_+,\ u_+=\inf S^+.
	$$
	
	Similarly, we produce
	$$
	v_-\in S^-\subseteq -D_+,\ v_-=\sup S^-.
	$$
\end{proof}

Let $\tau>||\xi||_\infty$ and consider the following truncation-perturbation of $\hat{f}(z,\cdot)$:
\begin{equation}\label{eq27}
	f_0(z,x)=\left\{
		\begin{array}{ll}
			\hat{f}(z,v_-(z)) + \tau|v_-(z)|^{p-2}v_-(z) & \mbox{if}\ x<v_-(z) \\
			\hat{f}(z,x)  + \tau|x|^{p-2}x & \mbox{if}\ v_-(z)\leq x\leq u_+(z) \\
			\hat{f}(z,u_+(z)) + \tau u_+(z)^{p-1} & \mbox{if}\ u_+(z)<x.
		\end{array}
	\right.
\end{equation}

We set $F_0(z,x)=\int_0^xf_0(z,s)ds$ and consider the $C^1$-functional $\varphi_0:W^{1,p}(\Omega)\rightarrow\RR$ defined by
$$
\begin{array}{rr}
	\varphi_0(u) = \int_\Omega G(Du)dz + \frac{1}{p}\int_\Omega[\xi(z)+\tau]|u|^pdz + \frac{1}{p}\int_\Omega\beta(z)|u|^pd\sigma - \int_\Omega F_0(z,u)dz \\
	\mbox{for all}\ u\in W^{1,p}(\Omega).
\end{array}
$$

Evidently, $\varphi_0(\cdot)$ is coercive (see (\ref{eq27}) and recall that $\tau>||\xi||_\infty$). So, $\varphi_0(\cdot)$ is bounded below and satisfies the PS-condition (see Marano \& Papageorgiou \cite{14, mara}).
\begin{prop}\label{prop8}
	If hypotheses $H(a),H(\xi),H(\beta),H(f)$ hold and $V\subseteq W^{1,p}(\Omega)$ is a finite dimensional linear subspace, then there exists $\rho_V>0$ such that
	$$
	\sup\left\{\varphi_0(u):u\in V,||u||=\rho_V\right\}<0.
	$$
\end{prop}
\begin{proof}
	Recall that $u_+\in D_+$ and $v_-\in-D_+$. So, $m_0=\min\{\min_{\overline\Omega}u_+,-\max_{\overline\Omega}v_-\}>0$. We set $\epsilon_0=\min\{\epsilon,m_0\}$ (where $\epsilon>0$ is from (\ref{eq4})). On account of hypothesis $H(f)(iii)$, given any $\mu>0$, we can find $\delta=\delta(\mu)>0\in(0,\epsilon_0)$ such that
	\begin{equation}\label{eq28}
		\begin{array}{ll}
			F_0(z,x)=\hat{F}(z,x) + \frac{\tau}{p}|x|^p & = F(z,x) + \frac{\tau}{p}|x|^p \\
			& \geq \frac{\mu}{q}|x|^q+\frac{\tau}{p}|x|^p \\
			& \mbox{(for almost all $z\in\Omega$, and all $|x|\leq\delta$, see (\ref{eq4}) and (\ref{eq27}))}.
		\end{array}
	\end{equation}
	
	Moreover, on account of hypothesis $H(a)(iv)$ and Corollary \ref{cor3}, we have
	\begin{equation}\label{eq29}
		G(y)\leq c_{10}[|y|^q+|y|^p]\ \mbox{for some}\ c_{10}>0,\ \mbox{and all}\ y\in\RR^N.
	\end{equation}
	
	Since the subspace $V\subseteq W^{1,p}(\Omega)$ is finite dimensional, all norms are equivalent. So, we can find $\rho_V\in(0,1]$ such that
	\begin{equation}\label{eq30}
		u\in V,\ ||u||\leq \rho_V\Rightarrow|u(z)|\leq\delta\ \mbox{for all}\ z\in\overline\Omega.
	\end{equation}
	
	Then for every $u\in V$ with $||u||\leq \rho_V$, we have
	$$
	\begin{array}{rr}
		\varphi_0(u)\leq c_{11}||u||^q - \mu c_{12}||u||^q\ \mbox{for some}\ c_{11}, c_{12}>0 \\
		\mbox{(see (\ref{eq27}), (\ref{eq28}), (\ref{eq29}), (\ref{eq30}) and recall that $\rho_V\leq1,\ q\leq p$)}
	\end{array}
	$$
	
	Since $\mu>0$ is arbitrary, we choose $\mu>\frac{c_{11}}{c_{12}}$ and conclude that
	$$\varphi_0(u)<0\ \mbox{for all}\ u\in V\ \mbox{with}\ ||u||=\rho_V.$$
The proof is now complete.
\end{proof}

We now obtain the following multiplicity theorem for the nodal solutions of problem (\ref{eq1}).
\begin{theorem}\label{th9}
	Assume that hypotheses $H(a), H(\xi), H(\beta), H(f)$ hold. Then there exists a sequence $\{u_n\}_{n\geq 1}\subseteq C^1(\overline\Omega)$ of nodal solutions of problem (\ref{eq1}) such that
	$$
	u_n\rightarrow0\ \mbox{in}\ C^1(\overline\Omega).
	$$
\end{theorem}
\begin{proof}
	We know that $\varphi_0(\cdot)$ is even, bounded below, satisfies the PS-condition, and $\varphi_0(0)=0$. Moreover, using (\ref{eq27}) as before, we can check that
	\begin{equation}\label{eq31}
		K_{\varphi_0}\subseteq[v_-,u_+]\cap C^1(\overline\Omega).
	\end{equation}
	
	The aforementioned properties of $\varphi_0(\cdot)$ and Proposition \ref{prop8} permit us to apply Theorem \ref{th1}. So, we can find a sequence $\{u_n\}_{n\geq1}\subseteq W^{1,p}(\Omega)$ such that
	\begin{equation}\label{eq32}
		u_n\in K_{\varphi_0}\subseteq [v_-,u_+]\cap C^1(\overline\Omega)\ \mbox{(see (\ref{eq31})) and}\ u_n\rightarrow0\ \mbox{in}\ W^{1,p}(\Omega).
	\end{equation}
	
	The nonlinear regularity theory of Lieberman \cite{8} implies that we can find $\gamma\in(0,1)$ and $c_{13}>0$ such that
	\begin{equation}\label{eq33}
		u_n\in C^{1,\gamma}(\overline\Omega),\ ||u_n||_{C^{1,\gamma}(\overline\Omega)}\leq c_{13}\ \mbox{for all}\ n\in\NN.
	\end{equation}
	
	We know that $C^{1,\gamma}(\overline\Omega)$ is  compactly embedded in $C^1(\overline\Omega)$. So, it follows from (\ref{eq32}) and (\ref{eq33}) that
	$$
	\begin{array}{ll}
		& u_n\rightarrow0\ \mbox{in}\ C^1(\overline\Omega),\\
		\Rightarrow & -\epsilon_0\leq u_n(z)\leq\epsilon_0\ \mbox{for all}\ z\in\overline\Omega,\ \mbox{and all}\ n\geq n_0 \\
		& \mbox{(recall that $\epsilon_0=\min\{\epsilon,m_0\}>0$, see the proof of Proposition \ref{prop8})}.
	\end{array}
	$$
	
	From (\ref{eq4}), (\ref{eq32}) and the extremality of $u_+,v_-$, we get that $\{u_n\}_{n\geq1}\subseteq C^1(\overline\Omega)$ are nodal solutions of (\ref{eq1}) and we have $u_n\rightarrow0$ in $C^1(\overline\Omega)$.
\end{proof}

\medskip
{\bf Acknowledgements.}  This research was supported by the Slovenian Research Agency grants
P1-0292, J1-8131, J1-7025, N1-0064, and N1-0083. V.D.~R\u adulescu acknowledges the support through a grant of the Romanian Ministry of Research and Innovation, CNCS--UEFISCDI, project number PN-III-P4-ID-PCE-2016-0130,
within PNCDI III.

\end{document}